\documentclass[11pt,reqno]{amsart} 
\usepackage[height=190mm,width=130mm]{geometry}
\usepackage{amsmath}
\usepackage{amssymb,latexsym}
\usepackage{amsrefs}
\usepackage[draft=true]{hyperref}
\usepackage{cite}
\usepackage{a4wide}
\usepackage{amscd}
\usepackage{graphics}
\usepackage{color}
\usepackage[all]{xy}

\theoremstyle{plain}
\newtheorem{theorem}{Theorem}
\newtheorem{lemma}{Lemma}
\newtheorem{corollary}{Corollary}
\newtheorem{proposition}{Proposition}

\theoremstyle{definition}

\theoremstyle{remark}
\newtheorem{remark}{Remark}

\newcommand{\Z}{\ensuremath{\mathbb{Z}}}   

\newcommand{\Hom}{\operatorname{Hom}}

\newcommand{\Ker}{\operatorname{{ \it ker}}}
\newcommand{\Rad}{\operatorname{{\it rad}}}
\newcommand{\Soc}{\operatorname{{\it soc}}}
\newcommand{\ds}{\subseteq ^{\oplus}}

\numberwithin{equation}{section} 

\begin{document}
\title[Simple-direct modules]{On  Simple-Direct  Modules}





\author{ENG\.{I}N B\"uy\"uka\c{s}{\i}k}

\address{Izmir Institute of Technology \\ Department of Mathematics\\ 35430 \\ Urla, \.{I}zmir\\ Turkey}

\email{enginbuyukasik@iyte.edu.tr}

\author{\"{O}zlem Demir}

\address{Izmir Institute of Technology \\ Department of Mathematics\\ 35430 \\ Urla, \.{I}zmir\\ Turkey}
\email{ozlemirmak@iyte.edu.tr}

\author{M\"{u}ge Diril}

\address{Izmir Institute of Technology \\ Department of Mathematics\\ 35430 \\ Urla, \.{I}zmir\\ Turkey}

\email{mugediril@iyte.edu.tr}

\begin{abstract} Recently, in a series of papers ``simple" versions of direct-injective and direct-projective modules have been investigated (see, \cite{sdi}, \cite{sdp1}, \cite{sdp2}). These modules are termed as ``simple-direct-injective" and ``simple-direct-projective", respectively.
In this paper, we give a complete characterization of the aforementioned modules over the ring of integers and over semilocal rings. The ring is semilocal if and only if every right module with  zero Jacobson radical is simple-direct-projective.  The rings whose simple-direct-injective right modules are simple-direct-projective are fully characterized. These are exactly the left perfect right $H$-rings. The rings whose simple-direct-projective right modules are simple-direct-injective are right max-rings. For a commutative Noetherian ring, we prove that simple-direct-projective modules are simple-direct-injective if and only if simple-direct-injective modules are simple-direct-projective if and only if the ring is Artinian. Various closure properties and some classes of modules that are  simple-direct-injective (resp. projective) are given.

\end{abstract}

\subjclass[2010]{16D50, 16D60, 18G25}

\keywords{Simple-direct-projective, simple-direct-injective, torsion submodule, perfect rings, H-rings, Artinian rings.}

\maketitle

\section{Introduction}

In \cite{nicholson1976}, a right module is called {\it direct-injective} if every submodule isomorphic to a summand is a summand. Direct-injective modules are also known as $C2$-modules. A right module is a $C3$-module if the sum of any two summands with zero intersection is again a summand. These modules and several generalizations are studied extensively in the literature. Recently, the "simple" version of $C2$ and $C3$-modules are studied in \cite{sdi}. Namely, a right module is called {\it simple-direct-injective} if every simple submodule isomorphic to summand is itself a summand, or equivalently if the sum of any two simple summands with zero intersection is again a summand (see, \cite{sdi}).

Dual to direct-injective modules, a right  module $M$ is called {\it direct-projective, or $D2$-module} if, for every submodule $A \subseteq M$ with $ \frac{M}{A}$ isomorphic to a summand of $M$, then $A$ is a direct summand of $M$ (see \cite{nicholson1976}).  In \cite{sdp1} and \cite{sdp2} the authors investigate and study a dual notion of simple-direct-injective modules. A right module $M$ is called \emph{simple-direct-projective} if, whenever $A$ and $B$ are submodules of $M$ with $B$ simple and $ \frac{M}{A}\cong B \ds M$, then $A \ds M$.
Some well known classes of rings and modules are characterized in terms of simple-direct-injective and simple-direct-projective modules (see, \cite{sdi}, \cite{sdp1}, \cite{sdp2}).

In this paper,  we characterize  simple-direct-injective and simple-direct-projective modules over the ring of integers and over semilocal rings. We show that, the ring is semilocal iff every right module with zero Jacobson radical is simple-direct-projective. We prove that the rings whose simple-direct-injective right modules are simple-direct-projective are exactly the left perfect right $H$-rings.  We show that, the rings whose simple-direct-projective modules are simple-direct-injective are right max-rings. For a commutative Noetherian ring, we prove that,  simple-direct-projective modules are simple-direct-injective iff simple-direct-injective modules are simple-direct-projective iff the ring is Artinian.

The paper is organized as follows.

In section 2, we characterize simple-direct-projective abelian groups (Theorem \ref{thm:sdpAbelianGroups}).
 As a byproduct a characterization of simple-direct-projective modules over local and local perfect rings is obtained. We prove that the ring is semilocal if and only if every right module with zero Jacobson radical is simple-direct-projective.

In section 3, a complete characterization of simple-direct-injective abelian groups is given (Theorem \ref{lemma:sdiabeliangroups}). Motivated by the fact that nonsingular right modules are simple-direct-projective over any ring, we prove the corresponding result for simple-direct-injective modules. We show that, nonsingular right modules are simple-direct-injective iff projective simple right modules are injective. We also give a characterization of simple-direct-injective modules over semilocal rings. We show that  simple-direct-injective modules are closed under coclosed submodules over any ring, and closed under pure submodules provided the ring is commutative. Partial converses of these results are given.

Following \cite[sec. 4.4]{injectivemodules}, we say $R$ is a {\it right $H$-ring} if for nonisomorphic simple right $R$-modules $S_1$ and $S_2$, $\Hom_R(E(S_1), E(S_2)) = 0$. Commutative Noetherian rings, and commutative semiartinian rings are $H$-ring by \cite[Proposition 4.21]{injectivemodules} and \cite[Proposition 2]{camillohrings}, respectively. Right Artinian rings that are right $H$-rings are characterized in \cite[Theorem 9]{papp}. Some classes of noncommutative $H$-rings are also studied in \cite{Golan}. A ring $R$ is called {\it right max-ring} if every nonzero right $R$-module has a maximal submodule.

In \cite[Theorem 3.4.]{sdi}, the authors characterize the rings over which simple-direct-injective right modules are $C3$. They prove that these rings are exactly the Artinian serial rings with $J^2(R)=0$. In \cite[Theorem 4.9.]{sdp1}, the authors prove that every simple-direct-injective right $R$-module is $D3$ iff every simple-direct-projective right $R$-module is $C3$ iff $R$ is uniserial with $J^2(R)=0$.

At this point, it is natural to consider the rings whose simple-direct-injective modules are simple-direct-projective, and the rings whose simple-direct-projective modules are simple-direct-injective. $C3$ and $D3$ right $R$-modules are simple-direct-injective and simple-direct-projective respectively. Thus, uniserial rings with $J^2(R)=0$ are examples of such rings.


In section 4,  we prove that, every simple-direct-injective right module is simple-direct-projective iff the ring is left perfect right $H$-ring (Theorem \ref{thm:charofsdisdp}). As a consequence, we show that, commutative perfect rings are examples of such rings. For a commutative Noetherian ring, we obtain that, simple-direct-injective modules are simple-direct-projective iff the ring is Artinian (Corollary \ref{cor:comm.Noethersdiaresdp}). We show that, the rings whose simple-direct-projective right modules are simple-direct-injective are right max-rings (Proposition \ref{prop:sdparesdithentheringMAXRING}). For a commutative Noetherian ring, we prove that, simple-direct-projective modules are simple-direct-injective iff simple-direct-injective modules are simple-direct-projective iff the ring is Artinian (Corollary \ref{cor:Artiniffsdisdpiffsdpsdi}).

Throughout, rings are associative with unity and modules are unitary. For a module $M$, we denote by $\Rad (M)$, $\Soc (M)$, $Z(M)$ and $E(M)$ the Jacobson radical, the socle, the singular submodule and the injective hull of $M$, respectively. The Jacobson radical of a ring $R$ will be denoted by $J(R)$. We write $L \subseteq M$ if $N$ is a submodule of $M$, and $L \ds M$
if $L$ is a direct summand of $M$. For a module $M$ over a commutative domain $R$, we denote the torsion submodule of $M$ by $T(M)$. Over the ring of integers, we denote by $\Omega$ the set of prime integers. It is well known that $T(M)=\oplus_{p \in \Omega}T_p(M)$, where $T_p(M)$ is the {\it p-primary component} of $T(M)$ i.e. the set of all $m \in T(M)$ such that $p^n.m=0$ for some positive integer $n$. An abelian group $G$ is {\it bounded} if $nG=0$ for some positive integer $n$. For $p \in \Omega$, the simple $\Z$-module of order $p$ will be denoted by $\Z_p$.  \\ A monomorphism $f: M \to N$ of right modules is called {\it pure-monomorphism} if the induced map $f \otimes 1_L : M \otimes L \to N \otimes L$ is a monomorphism for each left module $L$. Let $B$ be a right module and $A \subseteq B$. A is called a pure submodule of $B$ if the inclusion map $i: A \to B$ is a pure monomorphism. A subgroup $A$ of an abelian group $B$ is pure iff $nA=A \cap nB$ for each integer $n$ (see \cite{fuchs}). A right module $E$ is called {\it pure-injective} if $E$ for every pure monomorphism $f:M \to N$ of right modules, any homomorphism $g:M \to E$ can be extended to a homomorphism $h:N \to E$ (see, \cite{FuchsAndSalce:ModulesOverNonNoetherianDomains}).

\section{Simple-direct-projective modules}

In this section, we give a complete characterization of simple-direct-projective modules over the ring of integers. As a byproduct we obtain a characterization of simple-direct-projective modules over local, and local right perfect rings. We also prove that, the ring is semilocal iff every right module with zero Jacobson radical is simple-direct-projective.

Following \cite{sdp1}, a right $R$-module $M$ is called \emph{simple-direct-projective} if, whenever $A$ and $B$ are submodules of $M$ with $B$ is simple and $\frac{M}{A}\cong B \ds M$, then $A \ds M$.

A submodule $K$ of a module $M$ is {\it small}, denoted as $K \ll M$, if $K+N=M$ implies $N=M$ for each $N \subseteq M$.
A submodule $L$ of a module $M$ is {\it coclosed} in $M$ if $\frac{L}{K} \ll \frac{M}{K}$ implies $K=L$ for each submodule $K$ of $L$ (see \cite{lifting}). \\Let $S$ be a simple submodule of a module $M$. It is easy to see that, $S\ll M$ or $S \ds M$. Thus any simple coclosed submodule is a direct summand. This fact will be used in the sequel.

In order to characterize simple-direct-projective abelian groups, we need several lemmas. We begin with the following.

\begin{lemma}\label{lemma:coclosedsdp} Let $M$ be a simple-direct-projective right module, and $L$ be a coclosed submodule of $M$. If $\Soc(M) \subseteq L$, then $L$ is simple-direct-projective.
\end{lemma}

\begin{proof} Let $L$ be a coclosed submodule of $M$. Suppose $\frac{L}{K} \cong S \ds L$, where $S$ is a simple submodule of $L$. Then  $S$ is a coclosed submodule of $M$ as well by \cite[3.7.(1)]{lifting}. As $S$ is a coclosed submodule of $M$, $S$ is not small in $M$. Thus $S \ds M$.
Since $L$ is a coclosed submodule of $M$, $\frac{L}{K}$ is a coclosed submodule of $\frac{M}{K}$ by \cite[3.7.(1)]{lifting}. Thus $\frac{L}{K}$ is not small in $\frac{M}{K}$, and so $\frac{L}{K} \oplus \frac{N}{K} \cong \frac{M}{K}$, for some submodule $N$ of $M$. Clearly, $L \cap N =K$ and $\frac{M}{N} \cong S \ds M$. Since $M$ is simple-direct-projective, $M=N \oplus B$ for some simple submodule $B$ of $M$. Using the fact that $\Soc (M) \subseteq L$ we get, by modular law, that $L=L\cap N \oplus B$ i.e. $L\cap N=K \ds L$. Hence $L$ is simple-direct-projective.
\end{proof}

\begin{lemma}\label{lemma:torsionsubmodulecoclosed} Let $G$ be an abelian group and $T(G)$ be the torsion submodule of $G$. Then $T(G)$ is a coclosed submodule of $G$.
\end{lemma}

\begin{proof} Set $T=T(G)$. By  \cite[Proposition 8.12]{FuchsAndSalce:ModulesOverNonNoetherianDomains}, $T$ is a pure submodule of $G$. In order to show that $T$ is a coclosed submodule of $G$, suppose $\frac{T}{A}$ is small in $\frac{G}{A}$ for some proper submodule $A$ of $G$, and let us obtain a contradiction. If $\frac{T}{A}$ has no maximal submodules, then $\frac{T}{A}$  is injective by \cite[pg.99 Ex.1 and Theorem 21.1]{fuchs}. Being small and injective implies $\frac{T}{A}=0$ i.e. $T=A$, a contradiction. Now suppose, there is a maximal submodule $L$ of $T$ such that $A \subseteq L \subseteq T.$ By \cite[Lemma 5.18]{Anderson-Fuller:RingsandCategoriesofModules} homomorphic images of small submodules is small, and hence $\frac{T}{L}$ is small in $\frac{G}{L}$. By \cite[Lemma 26.1(ii)]{fuchs} pure subgroups are closed under factor modules, so $\frac{T}{L}$ is pure in $\frac{G}{L}$. On the other hand, $\frac{T}{L}$ is simple, and so it is bounded. Then $\frac{T}{L}$ is a direct summand of $\frac{G}{L}$ by \cite[Theorem 27.5]{fuchs}. Now $\frac{T}{L}$ is both  small and a direct summand in $\frac{G}{L}$, which is a contradiction. In conclusion $\frac{T}{A}$ is not small in $\frac{G}{A}$ for any proper subgroup $A \subseteq T$, that is,  $T$ is a coclosed subgroup of $G$.
\end{proof}

\begin{corollary}\label{corollary:torsionsubmoduledirectprojective} If $M$ is a simple-direct-projective abelian group, then the torsion submodule $T(M)$ of $M$ is simple-direct-projective.
\end{corollary}

\begin{proof} Let $M$ be a simple-direct-projective abelian group. Since simple abelian groups are torsion, $\Soc(M) \subseteq T(M)$. Then the proof is clear by Lemma \ref{lemma:coclosedsdp} and Lemma \ref{lemma:torsionsubmodulecoclosed}.
\end{proof}

The right modules with no simple summands, and the right modules whose maximal submodules are direct summands are trivial examples of simple-direct-projective modules. We include the following lemma for easy reference.

\begin{lemma}\label{lemma:extremesocle} Let $M$ a right module. Suppose $\Soc(M) \subseteq \Rad (M)$ or $\frac{M}{\Soc(M)}$ has no maximal submodules. Then $M$ is simple-direct-projective.
\end{lemma}

\begin{proof} If $\Soc(M) \subseteq \Rad (M)$, then $M$ has no simple summands and so it is simple-direct-projective. Now assume that $\frac{M}{\Soc (M)}$ has no maximal submodules, and let $K$ be a maximal submodule of $M$. Then $K + \Soc (M)=M$. Thus there is a simple submodule $S$ of $M$ such that $K +S=M$. By simplicity of $S$, $K\cap S=0$, and so $K \ds M$. Hence $M$ is simple-direct-projective.
\end{proof}

First, we give a characterization of simple-direct-projective torsion abelian groups.

\begin{proposition}\label{lemma:torsinsdp} Let $M$ be a torsion abelian group. The following statements are equivalent.
\begin{enumerate}
\item $M$ is simple-direct-projective.
\item $T_p(M)$ is simple-direct-projective for every $p \in \Omega$.
\item For every $p \in \Omega$,\\ $(i)$ $\Soc(T_p(M)) \subseteq \Rad (T_p(M))$, or\\
$(ii)$ $\displaystyle \frac{T_p(M)}{\Soc(T_p(M))}$ has no maximal submodules.
\end{enumerate}
\end{proposition}

\begin{proof} $(1) \Rightarrow (2)$ Since $M$ is torsion, $M=\oplus_{p \in \Omega}T_p(M)$. Then, by \cite[Proposition 2.4]{sdp1},  $T_p(M)$ is simple-direct-projective for every $p \in \Omega$.

$(2) \Rightarrow (3)$  Suppose $(i)$ does not hold. Then there is a simple subgroup $S$ of $T_p(M)$ such that $S$ is not contained in $\Rad(T_p(M))$. Thus $S$ is not small in $T_p(M)$, and so $S \ds T_p(M)$. Note that, all simple subgroups and simple factors of $T_p(M)$ are isomorphic to $S$. Assume that $A$ is a maximal submodule of $T_p(M)$ such that, $\Soc(T_p(M)) \subseteq A \subseteq T_p(M)$. Then $\frac{T_p(M)}{A}\cong S \ds T_p(M)$. Then, as $T_p(M)$ is simple-direct-projective, $T_p(M)=A \oplus S'$ for some simple submodule $S'$ of $T_p(M)$. Then $S' \subseteq \Soc(T_p(M)) \subseteq A$, which is a contradiction. Hence $\frac{T_p(M)}{\Soc(T_p(M))}$ has no maximal submodules i.e. $(ii)$ holds.

$(3) \Rightarrow (2)$ By Lemma \ref{lemma:extremesocle}.

$(2) \Rightarrow (1)$ Let $A$ and $B$ be subgroups of $M$ with $B$ simple and $\frac{M}{A} \cong B \ds M$. As $B$ is simple, there is a $p \in \Omega$ such that $B \subseteq T_p(M)$ and $pB=0$. As $B \ds M$, $B \ds T_p(M)$. Since $pB=0$ and $\frac{M}{A} \cong B$, we have $p(\frac{M}{A})=0$ i.e. $pM \subseteq A$. For any prime $q \neq p$, it is easy to see that,  $ T_q(M)=p T_q(M) \subseteq pM$. Thus for all primes $q \neq p$, $T_q(M) \subseteq pM \subseteq A$. Since $A$ is maximal subgroup, $T_p(M)$ is not contained in $A$, otherwise we would have $M=\oplus_{p \in \Omega}T_p(M) \subseteq A$, which is not the case as $A$ is a maximal subgroup of $M$. Thus, by maximality of $A$, we have $A+T_p(M)=M$. Then $$\frac{M}{A}=\frac{T_p(M)+A}{A} \cong \frac{T_p(M)}{A\cap T_p(M)} \cong B \ds T_p(M).$$ Since $T_p(M)$ is simple-direct-projective, $A\cap T_p(M) \oplus C=T_p(M)$ for some simple subgroup $C$ of $T_p(M)$. Then we get $M=A +T_p(M)=A+ [A\cap T_p(M) \oplus C]=A \oplus C$. Hence $M$ is simple-direct-projective.
\end{proof}

\begin{theorem}\label{thm:sdpAbelianGroups} Let $M$ be an abelian group. The following statements are equivalent.
\begin{enumerate}
\item $M$ is simple-direct-projective.
\item $(a)$ $T(M)$ is simple-direct-projective, and \\ $(b)$ for each $p \in \Omega$ such that $pM +T(M) \neq M$, $\Soc(T_p(M)) \subseteq \Rad (T_p(M))$.
\end{enumerate}
\end{theorem}

\begin{proof}$(1) \Rightarrow (2)$ By Corollary \ref{corollary:torsionsubmoduledirectprojective}, $T(M)$ is simple-direct-projective. Now let $p \in \Omega$ be such that $pM+ T(M) \neq M$. Then, as $\frac{M}{pM}$ is a homogoneous semisimple with each simple subgroup isomorphic to $\Z_p$ and $\frac{pM+T(M)}{pM}\neq \frac{M}{pM}$, there is a maximal subgroup $A$ of $M$ such that $T(M) \subseteq pM +T(M) \subseteq A$ and  $\frac{M}{A} \cong \Z_p$. \\ We need to show that, $\Soc(T_p(M)) \nsubseteq \Rad (T_p(M)$.  Suppose the contrary that $\Soc(T_p(M)) \nsubseteq \Rad (T_p(M)$. Then  there is a simple subgroup $S$ of $T_p(M)$ which is not contained in $\Rad (T_p(M)$. Then $S \ds T_p(M)$, and since $T_p(M)$ is a direct summand of $T(M)$, $S \ds T(M)$ as well. Then as $S$ is a pure subgroup of $T(M)$ and $T(M)$ is pure subgroup of $M$, $S$ is a pure subgroup of $M$. Thus $S$ is a pure and bounded subgroup of $M$, and so $S$ is a direct summand of $M$ by \cite[Theorem 27.5]{fuchs}. Since $S \cong \Z_p$ and $\frac{M}{A} \cong \Z_p \cong S \ds M$, simple-direct-projectivity of $M$ implies that $A \ds M$, i.e. $M=A \oplus D$ for some simple subgroup $D$ of $M$. Then $D \subseteq T(M) \subseteq A$, which is a contradiction. Hence we must have $\Soc(T_p(M)) \subseteq \Rad (T_p(M))$, and this proves $(2)$.

$(2) \Rightarrow (1)$ Let $A$ and $B$ be subgroups of $M$ with $B$ simple and $\frac{M}{A} \cong B \ds M$. Since $B$ is simple, $B \cong \Z_p$ for some $p \in \Omega$, in particular $B \subseteq \Soc(T_p(M)$ and $p(\frac{M}{A})\cong pB=0$ i.e. $pM \subseteq A$. As $B\ds M$, $B$ is not contained in $\Rad(T_p(M))$. Thus $\Soc(T_p(M)) \nsubseteq \Rad (T_p(M))$. Then $pM +T(M)=M$ by $(2)$. Thus $A +T(M)=M$. By similar arguments as in the proof of [Proposition \ref{lemma:torsinsdp}, $(2) \Rightarrow (1)$], we obtain that $A$ is a direct summand of $M$. Hence $M$ is simple-direct-projective.
\end{proof}

\begin{corollary} Let $M$ an abelian group. Suppose $\frac{M}{T(M)}$ has no maximal subgroups. Then $M$ is simple-direct-projective iff every maximal submodule of $M$ is a direct summand.
\end{corollary}

\begin{proof} Sufficiency is clear. To prove the necessity, let $A$ be a maximal subgroup of $M$. Suppose $\frac{M}{A}\cong \Z_p$, where $p \in \Omega$. Then $pM \subseteq A$. Since $\frac{M}{T(M)}$ has no maximal subgroups and $A$ is maximal, $A+T(M)=M$. Now, by the proof of [Theorem \ref{thm:sdpAbelianGroups}, $(2)\Rightarrow (1)$], $A \ds M$. This completes the proof.
\end{proof}

Over local rings, simple-direct-projective modules are exactly the modules given in Lemma \ref{lemma:extremesocle}.

\begin{proposition}\label{prop:sdpoverlocalrings} Let $R$ be a local ring. A right module $M$ is simple-direct-projective iff \\ $(i)$ $\Soc(M) \subseteq \Rad (M)$, or \\ $(ii)$ $\frac{M}{\Soc (M)}$ has no maximal submodules.
\end{proposition}
\begin{proof}
Suppose $(i)$ does not hold. Then there is a simple submodule $S$ of $M$ such that $M=N \oplus S$. Let $K$ be a maximal submodule of $M$. Since $R$ is a local ring, $R$ has a unique simple module up to isomorphism. Thus $\frac{M}{K}\cong S \ds M$. Hence simple-direct projectivity of $M$ implies that,  $K\ds M$. Thus any maximal submodule of $M$ is a direct summand. Now, if $L$ is a maximal submodule of $M$, such that $\Soc (M) \subseteq L \subseteq M$, then $M=L\oplus S'$ with $S'$ a simple submodule of $M$. Then $S' \subseteq \Soc (M) \subseteq L$, a contradiction. Hence $\frac{M}{\Soc(M)}$ has no maximal submodules. This proves the necessity.

Sufficiency is clear by Lemma \ref{lemma:extremesocle}.

\end{proof}

Over a right perfect ring, every module has a maximal submodule (\cite[Theorem 28.4]{Anderson-Fuller:RingsandCategoriesofModules}). Hence the following is a consequence of Proposition \ref{prop:sdpoverlocalrings}.

\begin{corollary} Let $R$ be a local right perfect ring. A right module $M$ is simple-direct-projective iff $M$ is semisimple or $\Soc (M) \subseteq \Rad (M)$.
\end{corollary}

It is easy to see that every module $M$ with $\Rad (M)=0$ is simple-direct-injective (see \cite[Remark 4.5]{sdp1}).  The following is the corresponding result for simple-direct-projective modules. Note that, a finitely generated module $M$ is semisimple iff every maximal submodule of $M$ is a direct summand.

\begin{proposition}\label{proposition:R semilocalzeroradicalaresdi}The following are equivalent for a ring $R$.
\begin{enumerate}
\item $R$ semilocal.
\item Every right module $M$ with $\Rad (M)=0$ is simple-direct-projective.
\item Every left module $M$ with $\Rad(M)=0$ is simple-direct-projective.
\end{enumerate}
\end{proposition}

\begin{proof} $(1) \Rightarrow (2)$ Since $R$ is semilocal, every right module with $\Rad(M)=0$ is semisimple (see \cite[Corollary 15.18]{Anderson-Fuller:RingsandCategoriesofModules}). Semisimple modules are simple-direct-projective. Thus $(2)$ follows.

$(2) \Rightarrow (1)$ Let $\mathcal{C}$ be a complete  set of representatives of nonisomorphic simple right modules. Let $N=\oplus_{S \in C}S$. Let $M$ be a right module with $\Rad (M)=0$. We shall prove that every maximal submodule of $M$ is a direct summand. Clearly, $\Rad (N\oplus M)=\Rad(N)\oplus \Rad(M)=0$, and so $N\oplus M$ is simple-direct-projective by $(2).$ Let $K$ be a maximal submodule of $M$. Then $N \oplus K$ is a maximal submodule of $N\oplus M$, and $\frac{N\oplus M}{N \oplus K} \cong S \ds N\oplus M$ for some simple submodule $S \in \mathcal{C}$. Thus $N\oplus K \ds N \oplus M$, i.e. $N \oplus M=(N \oplus K) \oplus L$ for some $L \subseteq N\oplus M$. By modular law, we get $M=K\oplus (M \cap(N \oplus L))$, i.e. $K \ds M$.  Now for the finitely generated module $M=\frac{R}{J(R)}$, we have $\Rad (\frac{R}{J(R)})=0$. Thus, by what we have shown, every maximal submodule of $\frac{R}{J(R)}$ is a direct summand. Hence $\frac{R}{J(R)}$ is semisimple, i.e. $R$ is semilocal.

$(1) \Leftrightarrow (3)$ Being semilocal is left right symmetric. Hence this is clear from the equivalence of $(1)$ and $(2)$.
\end{proof}

\section{Simple-direct-injective modules}

In this section, we give a characterization of simple-direct-injective modules over the ring of integers and over semilocal rings. Nonsingular right modules are simple-direct-projective over any ring (\cite[Example 2.5(2)]{sdp1}). Motivated by this fact, we  obtain a characterization of the rings whose nonsingular right modules are simple-direct-injective.

Following \cite{sdi}, a right module $M$ is called {\it simple-direct-injective} if, whenever $A$ and $B$ are simple submodules of $M$ with $A\cong B$ and $B\ds M$ we have $A \ds M$.

The following lemma is well-known. We do not know a proper reference, we include the proof for completeness.

\begin{lemma}\label{lemma:pure-injectivequotient rings} Let $R$ be a ring and $I$ a two sided ideal of $R$. Then any pure-injective right $\frac{R}{I}$-module is pure-injective as an $R$-module.
\end{lemma}

\begin{proof} Let $M$ be a pure-injective right $\frac{R}{I}$-module. Let $B$ be a right $R$-module, and $A$ a pure submodule of $B$. Let $i: A\to B$ be the inclusion map.  Then by \cite[Corollary 4.92]{Lam:LecturesOnModulesAndRings} $AI=A \cap BI$. Thus the natural map  $j: \frac{A}{AI} \to \frac{B}{BI}$ given by $j(a +AI)=a + BI$ is a pure monomorphism. In order to show that $M$ is a pure-injective $R$-module, let $f: A \to M$ be an $R$-homomorphism. Then $f(AI)=f(A)I \subseteq MI=0$. Thus $AI \subseteq \Ker(f)$, and so $f= \overline{f} \pi$, where $\pi :A \to \frac{A}{AI}$ is the natural epimorphism, and $\overline{f}: \frac{A}{AI} \to M$ is the homomorphism induced by $f$, i.e. $\overline{f}(a +AI)=f(a)$ for each $a \in A$. Since $M$ is a pure-injective $\frac{R}{I}$-module, there is homomorphism $g: \frac{B}{BI} \to M$ such that $\overline{f}=gj$. Let $\pi ' : B \to \frac{B}{BI}$ be the natural epimorphism. For $\phi = g \pi '$, it is straightforward to check that, $\phi i=f$, i.e. $\phi$ extends $f$ and so $M$ is a pure-injective $R$-module.
\end{proof}



\begin{lemma}\label{lemma:puresubmodules.d.i.} Let $R$ be a commutative ring. Let $M$ be an $R$-module and $N$ a pure submodule of $M$. If $M$ is simple-direct-injective, then $N$ is simple-direct-injective. The converse is true if $\Soc(M) \subseteq N$.
\end{lemma}

\begin{proof} Suppose $M$ is a simple-direct-injective module and $N$ a pure submodule of $M$. Let $S_1 \cong S_2$ with $S_1, \,S_2$ simple submodules of $N$ and $S_1 \ds N$. Now $S_1$ is pure in $N$, and $N$ is pure in $M$. Then $S_1$ is pure in $M$ by \cite[pages:39 and 43]{FuchsAndSalce:ModulesOverNonNoetherianDomains}. Since $R$ is commutative, simple modules are pure-injective by \cite[Corollary 4]{characterofsimple}.  Being pure and pure-injective implies $S_1\ds M$. Therefore $S_2 \ds M$, because $M$ is simple-direct-injective. Hence $S_2 \ds N$, and so $N$ is simple-direct-injective.

Now assume that, $N$ is a pure submodule of $M$, and $\Soc (M) \subseteq N$. Let $S_1 \cong S_2$ be two simple submodules of $M$ and $S_1 \ds M$. Then $S_1 \subseteq N$, $S_2 \subseteq N$ and $S_1 \ds N$. Since $N$ is simple-direct-injective, $S_2 \ds N$. As $S_2$ is pure in $N$ and $N$ is pure in $M$, $S_2$ is pure in $M$. Then $S_2 \ds M$, because $S_2$ is both pure-injective and pure in $M$. Hence $M$ is simple-direct-injective.
\end{proof}

A right module $M$ is called {\it absolutely pure} if it is pure in every module containing it as a submodule.

\begin{corollary} Let $R$ be a commutative ring and $M$ be an absolutely pure module. Then each module $K$ such that $M \subseteq K \subseteq E(M)$ is simple-direct-injective. In particular, absolutely pure modules are simple-direct-injective.
\end{corollary}

\begin{proof} Since $M$ is a pure submodule of $E(M)$ and $E(M)$ is simple-direct-injective, $M$ is simple-direct-injective by Lemma \ref{lemma:puresubmodules.d.i.}. As $M$ is essential in $E(M)$, $\Soc(M)=\Soc(K)$ for each module $K$ such that $M \subseteq K \subseteq E(M)$. Then $K$ is simple-direct-injective, again by Lemma \ref{lemma:puresubmodules.d.i.}.
\end{proof}

A commutative domain $R$ is called Pr\"{u}fer domain if each finitely generated ideal of $R$ is projective.

\begin{corollary}\label{cor:prufertorsionsdiiffMissdi} Let $R$ be a Pr\"{u}fer domain. A module $M$ is simple-direct-injective iff the torsion submodule $T(M)$ of $M$ is simple-direct-injective.
\end{corollary}

\begin{proof} Let $M$ be an $R$-module. Then $T(M)$ is pure in $M$ by \cite[Proposition 8.12]{FuchsAndSalce:ModulesOverNonNoetherianDomains}. Since simple modules are torsion, $\Soc (M) \subseteq T(M)$. Now the proof is clear by Lemma \ref{lemma:puresubmodules.d.i.}.
\end{proof}

\begin{lemma}\label{lemma:coclosedsubmodules.d.i.} Let  $M$ be an $R$-module and $N$ a coclosed submodule of $M$. If $M$ is simple-direct-injective, then $N$ is simple-direct-injective. The converse is true if $\Soc(M) \subseteq N$.
\end{lemma}

\begin{proof} Suppose $M$ is simple-direct-injective and $N$ a coclosed submodule of $M$. Suppose $S_1 \cong S_2$ are simple submodules of $N$ and $S_1 \ds N$. Then $S_1$ is a coclosed submodule of $M$ by \cite[3.7. (6)]{lifting}. Thus $S_1$ is not small in $M$, and so $S_1 \ds M$. By simple-direct-injectivity of $M$, $S_2 \ds M$. Therefore $S_2 \ds N$, and $N$ is simple-direct-injective.

Now assume that, $N$ is a coclosed submodule of $M$, and $\Soc (M) \subseteq N$. Let $S_1 \cong S_2$ be two simple submodules of $M$ and $S_1 \ds M$. Then $S_1 \subseteq N$, $S_2 \subseteq N$ and $S_1 \ds N$. Since $N$ is simple-direct-injective, $S_2 \ds N$. As $S_2$ is coclosed in $N$ and $N$ is coclosed in $N$, $S_2$ is coclosed in $M$. Then $S_2 \ds M$, and so $M$ is simple-direct-injective.
\end{proof}

\begin{theorem} \label{lemma:sdiabeliangroups} Let $M$ be an abelian group. The following statements are equivalent.
\begin{enumerate}
\item[(1)] $M$ is simple-direct-injective.
\item[(2)] $T(M)$ is simple-direct-injective.
\item[(3)] $T_p(M)$ is simple-direct-injective for each $p \in \Omega$.
\item[(4)] For each  $p \in \Omega$, $T_p(M)$ is semisimple, or $\Soc(T_p(M)) \subseteq \Rad(T_p(M))$.

\end{enumerate}
\end{theorem}

\begin{proof}
$(1) \Leftrightarrow (2)$ By Corollary \ref{cor:prufertorsionsdiiffMissdi}.




$(2) \Rightarrow (3)$ is clear, since $T(M)=\oplus_{p \in \Omega} T_p(M)$ and simple-direct-injective modules are closed under direct summands.

$(3) \Rightarrow (4)$ Assume that $\Soc(T_p(M)) \nsubseteq \Rad(T_p(M))$ for some  $p \in \Omega$. Then there is a simple subgroup $S$ of $T_p(M)$ such that $S \ds T_p(M)$. Let $A$ be the sum of all simple summands of $T_p(M)$. Then any finitely generated submodule of $A$ is a direct summand (hence  pure subgroup) of $T_p(M)$ by \cite[Lemma 2.4 (1)]{sdi}. Since $A$ is a direct limit of its finitely generated subgroups and direct limit of pure subgroups is pure, $A$ is pure in $T_p(M)$. As $A$ is semisimple and $A \subseteq T_p(M)$, $pA=0$ i.e. $A$ is bounded. Then $A \ds T_p(M)$ by \cite[Theorem 27.5]{fuchs}. Let $T_p(M)=A \oplus B$. We claim that $B=0$. If $B \neq 0$, then $\Soc (B)\neq 0$. Let $U$ be a simple subgroup of $B$. Since $T_p(M)$ is a $p$-group, $\Soc (T_p(M))$ is homogeneous i.e. all simple subgroups of $T_p(M)$ are isomorphic. Thus $U \ds T_p(M)$. Then $U \subseteq A$, which is a contradiction. Therefore $B=0$, and so $T_p(M)=A$ is semisimple. This proves $(4)$.

$(4) \Rightarrow (2)$ Let $U$ and $V$ be simple submodules of $T(M)$ such that $U \cong V$ and $U \ds $. Then there is a $p \in \Omega$ such that $U \ds T_p(M)$. Thus $T_p(M)$ must be semisimple by $(3)$. Since $V \cong U$, $V \subseteq T_p(M)$. Hence $V \ds T(M)$, and so $T(M)$ is simple-direct-injective.
\end{proof}

Recall that, a ring $R$ is { \it semilocal} if $\frac{R}{J(R)}$ is semisimple Artinian.

\begin{proposition}\label{thm:semilocalsimpledirectinjective}
Let $R$ be a semilocal ring. For a right $R$-module $M$, let $S'$ be sum of all simple summands of $M$. The following are equivalent.
\begin{enumerate}
  \item $M$ is simple-direct-injective.
  \item $S'$ is fully invariant and pure submodule of $M$.
  \item $M=S' \oplus N$, and $S'$ is a fully invariant submodule of $M$.
\end{enumerate}
\end{proposition}

\begin{proof}
$(1) \Rightarrow (2)$  By \cite[Lemma 2.4(2)]{sdi}, $S'$ is a fully invariant submodule of $M$. Let $S'=\oplus_{i \in I}V_i$, where $V_i$ are simple for each $i \in I$. Then for each finite subset $F \subseteq I$, $N_F=\oplus_{i \in F}V_i$ is a direct summand of $M$ by \cite[Lemma 2.4(1)]{sdi}, and so $N_F$ is a pure submodule of $M$. By \cite[4.84.(a)]{Lam:LecturesOnModulesAndRings} direct limit of pure submodules is pure, and so $S'=\oplus_{i \in I}V_i =\lim_{F}{N_F}$ is a pure submodule of $M$. This proves $(2).$

$(2) \Rightarrow (3)$ Since $R$ is a semilocal ring, $\frac{R}{J(R)}$ is semisimple. Thus every right $\frac{R}{J(R)}$-module is pure-injective. As $S'$ is semisimple, $S'.J(R)=0$. Thus $S'$ is a pure-injective right $R$-module by Lemma \ref{lemma:pure-injectivequotient rings}. Being pure and pure-injective implies that $S' \ds M$.

$(3) \Rightarrow (1)$ Let $A$ and $B$ be two simple submodules of $M$ such that $A\cong B$ and $A \ds M$. Then $A \subseteq S'$. Since $S'$ is a fully invariant submodule of $M$, $B \subseteq S'$ and so $B \ds M$. Hence $M$ is simple-direct-injective.
\end{proof}

Simple submodules of nonsingular modules are projective. Thus nonsingular right modules are simple-direct-projective over any ring (see \cite[Example 2.5]{sdp1}). The corresponding result for simple-direct-injective modules follows.

\begin{proposition}\label{prop:nonsingularsdi} Let $R$ be a ring. The following statements are equivalent.
\begin{enumerate}
  \item[(1)] Every projective simple right module is injective.
  \item[(2)] Every nonsingular right module is simple-direct-injective.
\end{enumerate}
\end{proposition}

\begin{proof} $(1) \Rightarrow (2)$ Nonsingular simple right modules are projective, and so injective by $(1)$. Thus $(2)$ follows.

$(2) \Rightarrow (1)$ Let $S$ be projective simple right module. Then $E(S)$ and $S\oplus E(S)$ are nonsingular, and so  $S\oplus E(S)$ is simple-direct-injective by $(2)$. Since $S\oplus 0 \cong 0 \oplus S$ and $S\oplus 0 \ds S \oplus E(S)$, $S \ds E(S)$. Hence $S$ is injective.
\end{proof}

\begin{corollary}
Let $R$ be a commutative ring. Then every nonsingular module is simple-direct-injective.
\end{corollary}

\begin{proof} Let $S$ be a projective simple module. Since $S$ is projective, it is flat. Then $S$ is injective by \cite[Lemma 2.6.]{ware}. Now, the conclusion follows by Proposition \ref{prop:nonsingularsdi}.
\end{proof}

Let $M$ be a right module and $N \subseteq M$. $N$ is called coneat submodule of $M$ if for every simple right module $S$, any homomorphism $f:N \to S$ can be extended to a homomorphism $g: M \to S$ (see, \cite{coneatBD}, \cite{coneatC}). A right module $M$ is called absolutely-coneat if $M$ is coneat in every module containing it as a submodule, equivalently $M$ is coneat in $E(M)$. It is easy to see that absolutely coneat modules are closed under direct summands, and that a simple right module is absolutely coneat iff it is injective.

\begin{proposition} Absolutely-coneat right modules are simple-direct-injective.
\end{proposition}

\begin{proof} Let $M$ be an absolutely-coneat right module. Suppose $A$ and $B$ are simple submodules of $M$ with $A\cong B$ and $B \ds M$. Then $B$ is absolutely-coneat as a direct summand of $M$. Thus $B$ is injective, and so $A$ is injective too. Then $A \ds M$, and hence $M$ is simple-direct-injective.
\end{proof}






\section{When simple-direct-injective (projective) modules are simple-direct-projective (injective)}

In \cite[Theorem 3.4.]{sdi}, the authors characterize the rings over which simple-direct-injective right modules are $C3$. They prove that these rings are exactly the Artinian serial rings with $J^2(R)=0$. In \cite[Theorem 4.9.]{sdp1}, the authors prove that every simple-direct-injective right $R$-module is $D3$ iff every simple-direct-projective right $R$-module is $C3$ iff $R$ is uniserial with $J^2(R)=0$.

At this point it is natural to consider the rings whose simple-direct-injective ( resp. projective) right modules are simple-direct-projective (resp. injective). Since $C3$ and $D3$ modules are simple-direct-injective and simple-direct-projective, respectively, uniserial rings with $J^2(R)=0$ are examples of the aforementioned rings.

In this section, we prove that every simple-direct-injective right module is simple-direct-projective iff the ring is left perfect and right $H$-ring. As a consequence, we show that, commutative perfect rings are examples of such rings.  We prove that the rings whose simple-direct-projective right modules are simple-direct-injective are right max-ring. For a commutative Noetherian ring, we prove that, simple-direct-projective modules are simple-direct-injective iff simple-direct-injective modules are simple-direct-projective iff the ring is Artinian.


Recall that, a ring $R$ is called {\it right semiartinian}  if every nonzero right $R$-module has nonzero socle. A right module $M$ is called { \it semiartinian (or Loewy)} module if every nonzero factor of $M$ has a nonzero socle.
First, we give a characterization of the rings over which every simple-direct-injective right module is simple-direct-projective. We begin with the following.

\begin{proposition}\label{prop:sdisdpimpliessemilocalsemiartin} Let $R$ be a ring. Suppose every simple-direct-injective right $R$-module is simple-direct-projective. Then $R$ is semilocal and right semiartinian, i.e. $R$ is left perfect.
\end{proposition}
\begin{proof} Every right module $M$ with $\Rad(M)=0$ is simple-direct-injective (see, \cite[Remark 4.5.]{sdp1} ). Thus, by Proposition \ref{proposition:R semilocalzeroradicalaresdi}, $R$ is semilocal. Suppose $R$ is not right semiartinian. Then there is a nonzero finitely generated right module $N$ with $\Soc(N)=0$. As the ring is semilocal, there are only finitely many, say $S_1,\,S_2,\cdots , S_n$ simple right modules up to isomorphism. Let $K=S_1 \oplus S_2 \oplus \cdots \oplus S_n \oplus N$. Then every simple submodule of $K$ is a direct summand, and so $K$ is simple-direct-injective. Let us show that $K$ is not simple-direct-projective, and get a contradiction. Let $L$ be a maximal submodule of $N$. Since $\Soc(N)=0$, $L$ is not a direct summand of $N$, and hence not a direct summand of $K$ too.
Let $L'=S_1 \oplus S_2 \oplus \cdots \oplus S_n \oplus L$. Then $L'$ is a maximal submodule of $K$, and so $\frac{K}{L'} \cong S_i \ds K$, for some $i=1,\cdots, n$. As $L$ is not a direct summand of $K$, $L'$ is not a direct summand of $K$ too. Thus $K$ is not simple-direct-projective, which is a contradiction. Therefore $R$ must be right semiartinian. Hence $R$ is left perfect by \cite[Theorem 23.20]{Lam:AFirstCourseInNoncommutativeRings}.
\end{proof}








\begin{theorem}\label{thm:charofsdisdp} The following statements are equivalent for a ring $R$.

\begin{enumerate}
\item $R$ is  left perfect and right $H$-ring.
\item Every simple-direct-injective right module is simple-direct-projective.
\end{enumerate}

\end{theorem}

\begin{proof}$(1) \Rightarrow (2)$ Let $M$ be a simple-direct-injective module. Let $A$ be the sum of all simple summands of $M$. Then $A$ is fully invariant and $M=A\oplus B$ by Proposition \ref{thm:semilocalsimpledirectinjective}. Since $A$ is a fully invariant submodule of $M$, $\Soc (B) \subseteq \Rad (M)$ and  $\Hom(A,\, \Soc (B))=0$. By $(1)$ the ring is right semiartinian, and so $\Soc (B)$ is an essential submodule of $B$. In order to prove that $M$ is simple-direct-projective, suppose that $\frac{M}{K} \cong S \ds M$ for some simple submodule $S$ of $M$. Then as $S \ds M$, $S \subseteq A$. We claim that, $A+ K= M$. Suppose the contrary that, $A+K$ is properly contained in $M$, and let us find a contradiction. Then, by maximality of $K$, we have $A\subseteq K$. Thus from $M=A \oplus B$ and by modular law, we get $K=A \oplus K\cap B$, and $$\frac{M}{K}=\frac{A\oplus B}{K}=\frac{A\oplus B}{A \oplus K\cap B}\cong \frac{B}{K\cap B}\cong S.$$
Thus $K\cap B$ is a maximal submodule of $B$. Set $N:= K\cap B$. Let $\Soc (B)=\oplus_{i \in I} U_i$, where $U_i$ is simple for each $i \in I$.
Since $\Soc(B)$ is essential submodule of $B$, the injective hull of $B$ is $E(B)=\oplus_{i \in I} E(U_i)$. As $\frac{B}{N} \cong S$, there is an epimorphism $f:B\rightarrow S$. Let $e: S\to E(S)$ be the inclusion homomorphism. Then the homomorphism $ef$ extends to a (nonzero) homomorphism $g:E(B) \to E(S)$. Since $E(B)=\oplus_{i \in I} E(U_i)$ and $g$ is nonzero, there is a nonzero homomorphism $h: E(U_j) \to E(S)$, for some $j \in I$. Then by the right $H$-ring assumption we must have $S \cong U_j$. Thus $\Hom(A, \Soc (B))\neq 0$, which is a contradiction. Hence the case $A+K=M$ must hold. Therefore, as $A$ is semisimple, there is a simple submodule $U$ of $A$ such that $U +K=M$ and $U \cap K=0$, i.e. $K \ds M$. Hence $M$ is simple-direct-projective. This proves $(2)$.

$(2) \Rightarrow (1)$ The ring $R$ is left perfect by Proposition \ref{prop:sdisdpimpliessemilocalsemiartin}. Suppose $R$ is not right $H$-ring. Then there are nonisomorphic simple right modules $S_1$ and $S_2$ such that $\Hom (E(S_1), \, E(S_2)) \neq 0$. Let $0 \neq f: E(S_1) \to E(S_2)$, and $A=\Ker(f)$. Since $\frac{E(S_1)}{A}\cong f(E(S_1)) \subseteq E(S_2)$, there is a submodule $B \subseteq E(S_1)$ such that $\frac{B}{A}\cong S_2$. Then it is clear that $B \oplus S_2$ is a simple-direct-injective right module. On the other hand, $\frac{B \oplus S_2}{A \oplus S_2}\cong 0 \oplus S_2 \ds B\oplus S_2$. But $A \oplus S_2$ is not a direct summand of $B \oplus S_2$. Thus $B \oplus S_2$ is not  simple-direct-projective. This contradicts $(2)$. Thus $R$ must be right $H$-ring.
\end{proof}

Now, we give some consequences of Theorem \ref{thm:charofsdisdp}.

\begin{corollary}\label{cor:commperfectdsdisdp} Let $R$ be a commutative ring. The following statements are equivalent.

\begin{enumerate}
\item $R$ is a perfect ring.
\item Every simple-direct-injective module is simple-direct-projective.
\end{enumerate}
\end{corollary}

\begin{proof} Commutative perfect rings are semiartinian (see, \cite[Theorem 23.20]{Lam:AFirstCourseInNoncommutativeRings}). Thus commutative perfect rings are $H$-ring by \cite[Proposition 2]{camillohrings}. Now the proof is clear by Theorem \ref{thm:charofsdisdp}.
\end{proof}



A right  Noetherian  right semiartian ring is right Artinian (see \cite{shock}). Left perfect rings are right semiartinian by \cite[Theorem 23.20]{Lam:AFirstCourseInNoncommutativeRings}. Thus the following is clear by Theorem \ref{thm:charofsdisdp}.

\begin{corollary}
Let $R$ be a right Noetherian ring. The following statements are equivalent.
\begin{enumerate}
\item $R$ is right Artinian right H-ring.

\item Every simple-direct-injective right module is simple-direct-projective.
\end{enumerate}
\end{corollary}

By \cite[Proposition 4.21]{injectivemodules}, commutative Noetherian rings are $H$-rings.

\begin{corollary}\label{cor:comm.Noethersdiaresdp}
   Let $R$ be a commutative Noetherian ring. The following statements are equivalent.
\begin{enumerate}

\item $R$ is Artinian ring.
\item Every simple-direct-injective module is simple-direct-projective.
\end{enumerate}

\end{corollary}

A ring $R$ is called  {\it right $V$-ring} if every simple right $R$-module is injective. By \cite[Theorem 4.1]{sdi}, $R$ is right $V$-ring iff every right $R$-module is simple-direct-injective.  A ring $R$ is called  {\it right max-ring} if every nonzero right $R$-module has a maximal submodule. Right $V$-rings are right max-rings (see \cite[Theorem 3.75]{Lam:LecturesOnModulesAndRings}). Clearly, over right $V$-rings simple-direct-projective right modules are simple-direct-injective.

Now, we consider the rings whose simple-direct-projective right modules are simple-direct-injective.

\begin{proposition}\label{prop:sdparesdithentheringMAXRING} Let $R$ be a ring. If each simple-direct-projective right $R$-module is simple-direct-injective, then $R$ is a right max-ring.
\end{proposition}

\begin{proof} Suppose the ring is not right max-ring. Then there is a nonzero right module $M$ such that $M=\Rad (M)$. Let $0\neq m \in M$, and let $K$ be a maximal submodule of $mR$. Let $h=i \pi: mR \to E(\frac{mR}{K})$, where $\pi :mR \to \frac{mR}{K}$ is the natural epimorphism and $i: \frac{mR}{K} \to E(\frac{mR}{K})$ is the inclusion homomorphism. By injectivity of $E(\frac{mR}{K})$, there is a (nonzero) homomorphism $g: M \to E \frac{mR}{K}$ which extends $h$. Let $L:=g(M)$. Since $\frac{M}{\Ker (g)}\cong L$ and $\Rad(M)=M$,   $L=\Rad(L)$. Note that $L$ has an essential socle isomorphic to $\frac{mR}{K}$. Consider the right module $N= \frac{mR}{K} \oplus L$. Then $0\oplus L$ is the unique maximal submodule of $N$ and $0\oplus L \ds N$. Thus $N$ is simple-direct-projective. On the other hand, $0 \oplus \Soc(L) \cong \frac{mR}{K} \oplus 0 \ds N$, but $0 \oplus \Soc(L)$ is not a direct summand of $N$. Therefore $N$ is not simple-direct-injective. This contradicts with our assumption that simple-direct-projective modules are simple-direct-injective. Hence $R$ must be right max-ring.
\end{proof}

A subfactor of a right module $M$, is a submodule of some factor module of $M$. The following lemma can be easily derived  from the definition of $H$-ring. We include it for an easy reference.

\begin{lemma}\label{lemma:subfactorofsimplesoverHrings}$R$ is a right $H$-ring iff for every simple right $R$-module $S$, every simple subfactor of $E(S)$ is isomorphic to $S$.
\end{lemma}

\begin{proof} Suppose $R$ is a right $H$-ring and $S$ a simple right $R$-module. Let $\frac{A}{B}$ be a simple subfactor of $E(S)$. Assume that $\frac{A}{B}$ is not isomorphic to $S$. Let $i_1: \frac{A}{B} \to \frac{E(S)}{B}$ and $i_2 : \frac{A}{B} \to E(\frac{A}{B})$ be the corresponding inclusions. Then there is a nonzero homomorphism $f: \frac{E(S)}{B} \to E(\frac{A}{B})$. Thus, $f\pi : E(S) \to E(\frac{A}{B})$ is a nonzero homomorphism, where $\pi : E(S) \to \frac{E(S)}{B}$ is the canonical epimorphism. This contradict with the assumption that $R$ is right $H$-ring. Therefore every simple subfactor of $E(S)$ is isomorphic to $S$. This proves the necessity.

Conversely let $S_1$ and $S_2$ be simple right $R$-modules and $0 \neq f \in \Hom_R (E(S_1), E(S_2))$. Then $\frac{E(S_1)}{\Ker(f)}$ has a simple subfactor isomorphic to $S_2$. Thus, by our assumption, we must have $S_1 \cong S_2$. Hence $R$ is a right $H$-ring.
\end{proof}

\begin{proposition}\label{prop:artinsdpsdi} Let $R$ be a commutative Noetherian ring. The following statements are equivalent.

\begin{enumerate}
  \item $R$ is Artinian.
  \item Every simple-direct-projective module is simple-direct-injective.
\end{enumerate}
\end{proposition}

\begin{proof} $(2) \Rightarrow (1)$ By Proposition \ref{prop:sdparesdithentheringMAXRING}, $R$ is a max-ring. Noetherian max-rings are Artinian by \cite[Theorem 1]{hamsher}.

$(1) \Rightarrow (2)$ Let $M$ be a simple-direct-projective $R$-module. Let $S'$ be the sum of simple summands of $M$. Then by the same arguments in the proof of [Proposition \ref{thm:semilocalsimpledirectinjective}, $(2)\Rightarrow (3)$] $S'$ is pure and  pure-injective submodule of $M$, and so $S' \ds M$. Let $M=S' \oplus N$. Clearly, by the construction of $S'$, $N$ has no simple and maximal summands. Now in order to prove that $M$ is simple-direct-injective, by Proposition \ref{thm:semilocalsimpledirectinjective}, it is enough to see that $S'$ is a fully invariant submodule of $M$. Suppose the contrary that there are  simple submodules $A, \, B$ of $M$ such that  $A \subseteq S'$, $B \subseteq N$ and $A \cong B$. Since $B \subseteq N$, there is a nonzero homomorphism $g: N \to E(B)$. Then for $K=\Ker(g)$, the module $\frac{N}{K}$ has a maximal submodule say $\frac{L}{K}$ by the Artinianity of $R$. Since $R$ is commutative and Noetherian,  $R$ is an $H$-ring. Thus every simple subfactor of $E(B)$ is isomorphic to $B$ by Lemma \ref{lemma:subfactorofsimplesoverHrings}. Therefore $\frac{N}{L} \cong B$. Now, $$\frac{M}{S'\oplus L}=\frac{S' \oplus N}{S' \oplus L} \cong B \cong A \ds M.$$ Then by simple-direct-projectivity of $M$, $S' \oplus L \ds M$ and, by modular law, $L\ds N$. This contradicts the fact that, $N$ has no maximal summands. Hence $S'$ is a fully invariant submodule of $M$, and so $M$ is simple-direct-injective by Proposition \ref{thm:semilocalsimpledirectinjective}. This proves $(2)$.
\end{proof}

\begin{proposition}\label{prop:commsemilocalsdparesdi} Let $R$ be a commutative semilocal ring. The following statements are equivalent.
\begin{enumerate}
  \item $R$ is perfect.
  \item Every simple-direct-projective module is simple-direct-injective.
\end{enumerate}
\end{proposition}

\begin{proof} $(2) \Rightarrow (1)$ $R$ is a max-ring by Proposition \ref{prop:sdparesdithentheringMAXRING}. Semilocal max-rings are perfect by \cite[Theorem 28.4]{Anderson-Fuller:RingsandCategoriesofModules}.

$(1) \Rightarrow (2)$ Note that, commutative perfect rings are H-rings and max-rings. Now, replacing Artinian ring by perfect ring the same proof of [Proposition \ref{prop:artinsdpsdi} $(1) \Rightarrow (2)$]  holds.
\end{proof}

\begin{remark}Over a right $V$-ring all right modules, in particular, simple-direct-projective right modules are simple-direct-injective (see \cite[Theorem 4.1]{sdi}). Since commutative perfect $V$-rings are semisimple, there is a simple-direct-injective $R$-module which is not simple-direct-projective over nonsemisimple commutative $V$-rings by Corollary \ref{cor:commperfectdsdisdp}. Therefore nonsemisimple commutative $V$-rings are examples of rings such that simple-direct-projective modules are simple-direct-injective, and admit a simple-direct-injective module that is not  simple-direct-projective.
\end{remark}

Summing up, Corollary \ref{cor:commperfectdsdisdp}, Corollary \ref{cor:comm.Noethersdiaresdp}, Proposition \ref{prop:artinsdpsdi} and Proposition \ref{prop:commsemilocalsdparesdi} we obtain the following.

\begin{corollary}\label{cor:Artiniffsdisdpiffsdpsdi} Let $R$ be a commutative Noetherian ring. Then the following statements are equivalent.
\begin{enumerate}
  \item $R$ is Artinian.
  \item Every simple-direct-injective module is simple-direct-projective.
  \item Every simple-direct-projective module is simple-direct-injective.
\end{enumerate}
\end{corollary}

\begin{corollary} Let $R$ be a commutative semilocal ring. Then the following statements are equivalent.
\begin{enumerate}
  \item $R$ is perfect.
  \item Every simple-direct-injective module is simple-direct-projective.
  \item Every simple-direct-projective module is simple-direct-injective.
\end{enumerate}
\end{corollary}


\begin{thebibliography}{00}



\bibitem{Anderson-Fuller:RingsandCategoriesofModules}
F. W. Anderson, K. R. Fuller, {\it Rings and categories of modules}, Springer-Verlag, New York, 1992.


\bibitem{coneatBD} E. B\"{u}y\"{u}ka\c{s}\i k, Y. Durgun,   Coneat submodules and coneat-flat modules, {\it Jour.  Korean Math. Soc.} {\bf 51 (6)} (1978), 1305-1319.


\bibitem{camillohrings} V. Camillo, Homological independence of injective hulls of simple modules over commutative rings, {\it Comm. Algebra} {\bf 6 (14)} (1978), 1459-1469.

\bibitem{sdi}
V. Camillo, Y. Ibrahim, M. Yousif, Y. Zhou,  Simple-direct-injective modules, {\it J. Algebra}, {\bf 420} (2014), 39-53.


\bibitem{lifting} J. Clark, C. Lomp, N. Vanaja, and R. Wisbauer, {\it Lifting modules}, Frontiers in Mathematics, Birkh$\ddot{a}$user Verlag, Basel, 2006.


\bibitem{characterofsimple}
T. J. Cheatham and J. R. Smith, Regular and semisimple modules, {\it Pacific J. Math.} {\bf 65} (1976), no. 2, 315-323.

\bibitem{coneatC} S. Crivei, Neat and coneat submodules over commutative rings, {\it Bull. Aust. Math. Soc} {\bf 89 (2)} (2014) 343-352.




\bibitem{fuchs}
L. Fuchs,  {\it Infinite abelian groups}, Academic Press, 1970.

\bibitem{FuchsAndSalce:ModulesOverNonNoetherianDomains} L. Fuchs, L. Salce, {\it Modules over non-noetherian domains}, American Math. Soc., 2000.


\bibitem{Golan} J. S. Golan, Homological Independence of Injective Hulls of Simple Modules Over Certain Noncommutative Rings, {\it Proc. Amer. Math. Soc.} {\bf 81-3}, (1981), 377-381.

\bibitem{Goodearl} K. R. Goodearl, {\it Ring theory}, Marcel Dekker, Inc., New York-Basel, 1976. Nonsingular rings and modules, Pure and Applied Mathematics, No. 33.

\bibitem{sdp1}
Y. Ibrahim, M. T. Ko\c{s}an, T. C. Quynh and M. Yousif, Simple direct-projective modules, {\it Comm. Algebra} {\bf 44} (2016), 5163-5178.

\bibitem{sdp2}
Y. Ibrahim, M. T. Ko\c{s}an, T. C. Quynh and M. Yousif, Simple-direct-modules, {\it Comm. Algebra} {\bf 45} (2017), 3643-3652.

\bibitem{hamsher} R. M. Hamsher, Commutative noetherian rings over which every module has a maximal submodule, {\it Proc. Amer. Math. Soc.} {\bf 17} (1966) 1471–1472.



\bibitem{Lam:LecturesOnModulesAndRings}
T. Y. Lam, {\it Lectures on modules and rings}, Springer-Verlag, New York, 1999.

\bibitem{Lam:AFirstCourseInNoncommutativeRings}
T. Y. Lam, {\it A first course in noncommutative rings}, Second, Graduate Texts in Mathematics, vol. 131, Springer-Verlag, New York, 2001.


\bibitem{ap}
C. Megibben, Absolutely pure modules,  {\it Proc.  American Math. Soc.}
{\bf 26-4}, (1970), 561-566.

\bibitem{nicholson1976} W. K. Nicholson, Semiregular modules and rings, {\it Canadian J. Math.} {\bf XXVIII(5)} (1976), 1105-1120.







\bibitem{papp} Z. Papp, On stable Noetherian rings, {\it Trans. Amer. Math. Soc.} {\bf 213} (1975), 107-114.


\bibitem{injectivemodules} D. W. Sharpe and P. Vamos, { \it Injective modules}, Cambridge Univ. Press, New York,
1972.

\bibitem{shock} R. C. Shock,  Dual generalizations of the Artinian and Noetherian conditions, {\it Pacific J. Math.}, {\bf 54},
   (1974), no. 2, 227-235,


\bibitem{ware} R. Ware, Endomorphism rings of projective modules,  \textit{Trans. Amer. Math. Soc.} {\bf 155} (1971), 233-256.


\bibitem{wisbauer} R. Wisbauer, {\it Foundations of module and ring theory}, Gordon and Breach Science Publishers, 1991.




\end{thebibliography}
\end{document}